\documentclass{amsart}
\usepackage{amsmath}
\usepackage{bibentry}
\usepackage{hyperref}
\usepackage{todonotes}
\usepackage{xcolor}
\usepackage{dynkin-diagrams}
\usepackage{dsfont}
\usepackage{romannum}
 \newtheorem{thm}{Theorem}[section]
\newtheorem{theorem}[thm]{Theorem}

\newtheorem{prop}[thm]{Proposition}
\newtheorem{proposition}[thm]{Proposition}
\newtheorem{lemma}[thm]{Lemma}

\newtheorem{rmk}[thm]{Remark}
\newtheorem{question}[thm]{Question}

\newtheorem{definition}[thm]{Definition}
\newtheorem{corollary}[thm]{Corollary}

\title[Comparing K\"{a}hler cone and symplectic cone]{Comparing K\"{a}hler cone and symplectic cone of one-point blowup of Enriques surface}

\author{Shengzhen Ning}
\AtEndDocument{\bigskip{\footnotesize
		\textsc{School of Mathematics, University of Minnesota, Minneapolis, MN, US} \par
		\textit{E-mail address}: \texttt{ning0040@umn.edu} \par	
}}

\begin{document}
\maketitle
\begin{abstract}
    We follow the study by Cascini-Panov (\cite{cascinipanov}) on symplectic generic complex structures on K\"{a}hler surfaces with $p_g=0$, a question proposed by Tian-Jun Li in \cite{spaceofsymp}, by demonstrating that the one-point blowup of an Enriques surface admits non-Kähler symplectic forms. This phenomenon relies on the abundance of elliptic fibrations on Enriques surfaces, characterized by various invariants from algebraic geometry. We also provide a quantitative comparison of these invariants to further give a detailed examination of the distinction between Kähler cone and symplectic cone.

\end{abstract}
\pagenumbering{arabic}
\tableofcontents
\section{Introduction}

Let $(X,\omega)$ be a closed symplectic manifold. It is called \emph{non-K\"{a}hler} if $\omega$ is not compatible with {\bf any} complex structure $J$ on $X$. In other words, $\omega(\cdot,J\cdot)$ doesn't define a Riemannian metric on $X$. Let's emphasize that the notion `non-K\"{a}hler' here should be understood to describe the symplectic manifold $(X,\omega)$ or the symplectic form $\omega$ rather than only the manifold\footnote{By manifold, we always mean a {\bf smooth oriented closed} manifold. The symplectic, complex or K\"{a}hler structure has to be compatible with the orientation.} $X$.

A well-known fact in symplectic geometry is that the space of compatible almost complex structures is always non-empty and contractible. Sometimes it's also worthwhile to investigate its subspace consisting of the integrable ones. For example, when $X$ is a rational ruled surface, \cite{AGKduke} shows that the inclusion of this subspace is a weak homotopy equivalence (in particular, non-empty) which plays an essential role in their study of the topology of the symplectomorphism groups. Nevertheless, it is possible for this subspace consisting of compatible complex structures to be empty for many other symplectic manifolds. The first such example is due to Thurston. In \cite{Thurston76}, he revisited Kodaira's example of a complex surface with odd first Betti number, which is topologically a torus bundle over torus. Since then more and more examples have been discovered with different techniques showing the non-K\"{a}hlerness in both dimension $4$ and higher dimensions. The first simply connected example was given by \cite{McDnonKahler} in higher dimension using symplectic blowup operation. Now, setting aside the variety of examples found in higher dimensions, we will concentrate on dimension 4 and the known examples\footnote{Obviously, it is unrealistic to give a complete list all the known examples here. I feel sorry for the authors whose examples are not mentioned here due to my limited knowledge. } can actually be classified (very roughly) into three categories. 
\begin{enumerate}
   \item \emph{The manifold $X$ admits no complex structure.} In this category, one can apply results from the theory of complex surface such as Kodaira's classification and Noether's inequality. For example, the solvmanifolds in \cite{solvmanifold} and the manifolds constructed from fiber sum operations in \cite{Gompffibersum}. When gauge theory came into play, there were more examples such as the homotopy K3 surfaces in \cite{GompfMrowka} using Donaldson's invariants or the knot surgery manifolds $E(1)_K$ for some fibered knot $K$ in \cite{Jparknonsympcomp} using Seiberg-Witten invariants, which share the feature that the smooth structures are exotic.

    \item \emph{The manifold $X$ can either admit complex structure or not, but it admits no K\"{a}hler structure.} The first example \cite{Thurston76} belongs to this category, where the topological constraint of even first Betti number was used. Other than the parity of Betti numbers, other topological constraints on K\"{a}hler manifolds such as formality, fundamental group and hard Lefschetz property can also be used to construct examples in this category. For instance, all nilmanifolds other than $T^4$ (see \cite{Rationalhomotopybook}), Gompf's symplectic manifolds with $\pi_1$ not satisfying the condition like \cite{JRpioneofkahler} and some $S^1$-bundles over mapping torus with degenerate cup product on $H^1$ constructed in \cite{Baldridgekod1,BaldLi}.

    \item \emph{The manifold $X$ admits K\"{a}hler structure.} In this case, one has to construct a symplectic form on the manifold underlying a K\"{a}hler surface and then argue it can not be the K\"{a}hler form of any complex structure. \cite{Draghici} showed for all $b_2^+>1$ minimal K\"{a}hler surface of general type, one can deform the K\"{a}hler form in the direction along the real part of holomorphic $2$-form to obtain a symplectic form which will be non-K\"{a}hler due to Hodge index theorem. More examples in this category include $T^4\#\overline{\mathbb{CP}^2}$ in \cite{T4}, and $(S^2\times T^2)\#\overline{\mathbb{CP}^2}$ and Burniat surface (of general type and $b_2^+=1$) in \cite{cascinipanov}.
\end{enumerate}

Note that subtlety exists in the above classification.  For instance, there are examples that deviate from the topological constraints of K\"{a}hler manifolds but with even first Betti number $b_1$. This would eventually imply $X$ admits no complex structure due to \cite{Buchdahl,Lamari}. However we are more willing to put it into category (2) rather than (1) since the constraints really come from the property of being K\"{a}hler. Moreover, the hard Lefschetz property, though we stated it in category (2), can be employed to build examples in category (3) in higher dimension. See \cite{ChoHardLef} for his example of an $S^2$-bundle over the K3 surface which admits K\"{a}hler structure and also symplectic forms of non-hard Lefschetz type. It's interesting to see whether one can still find such an example in dimension $4$.

It seems that examples in the category (3) are the most limited among the three. The purpose of this paper is to add another example into this category, which is the one point blowup of Enriques surface as suggested by the title. The motivation for studying such an example is that the authors of \cite{cascinipanov} forgot to add the minimal assumption in their statement of Proposition 4.2 (although they added the assumption in their Question 4.1). Moreover, in the end of \cite{cascinipanov} the authors were asking for an example with finite fundamental group. And Enriques surface has fundamental group $\mathbb{Z}_2$.

Our approach of showing non-K\"{a}hlerness is in the same spirit of \cite{T4} and \cite{cascinipanov}. By combining results from both algebraic geometry and symplectic geometry, we can compare the K\"{a}hler cone with the symplectic cone which we now define. For a manifold $X$, we define its \emph{symplectic cone} to be 
\[\mathcal{C}(X):=\{x\in H^2(X;\mathbb{R})\,|\,x\text{ is represented by a symplectic form}\}.\]If $J$ is a complex structure on $X$, then the $J$-\emph{K\"{a}hler cone} and \emph{K\"{a}hler cone} are defined to be 
\[\mathcal{K}_J(X):=\{x\in H^2(X;\mathbb{R})\,|\,x\text{ is represented by a symplectic form compatible with }J\},\]
\[\mathcal{K}(X):=\bigcup_{\text{all complex structures } J} \mathcal{K}_J(X).\]
We use $S$ to denote the Enriques surface and $\tilde{S}$ its blowup. Then our main result can be formulated by the following.
\begin{theorem}\label{thm:mainthm1}
   $\mathcal{C}(\tilde{S})\neq \mathcal{K}(\tilde{S}).$
\end{theorem}

\begin{rmk}
Obviously, the distinction between K\"{a}hler cone and symplectic cone implies the existence of non-K\"{a}hler symplectic form. However it remains unclear whether the equality of two cones would imply all symplectic forms are K\"{a}hler. For Kodaira dimension $-\infty$ manifolds, this implication holds true since all cohomologous symplectic forms are diffeomorphic (see \cite{spaceofsymp}). For Kodaira dimension $\geq 0$ manifolds, less is known but there is a speculation by Donaldson that one can attack this problem via almost K\"{a}hler Calabi-Yau equation (see \cite{Donaldsontwoform}).
\end{rmk}

Our example has $b_2^+=1$. Under such circumstances, the symplectic cone and each $J$-K\"{a}hler cone are both open in $H^2(X;\mathbb{R})$. Moreover, \cite{LiLiucone} shows that $\mathcal{C}(X)$ is a disjoint union of $\mathcal{C}_K(X)$ for all $K\in H^2(X;\mathbb{Z})$, where 
\[\mathcal{C}_K(X):=\{x\in \mathcal{C}(X)\,|\,x\text{ is represented by a symplectic form }\omega \text{ with }c_1(X,\omega)=-K\}.\]
\cite{spaceofsymp} introduced the concept of the so-called \emph{symplectic generic complex structure}, which just means the complex structure $J$ with its canonical class $K_J$ satisfies $\mathcal{C}_{K_J}(X)=\mathcal{K}_J(X)$. \cite[Question 4.6]{spaceofsymp} further asked about the existence of a symplectic generic complex structure on a $b_2^+=1$ manifold admitting K\"{a}hler structure. Since we prove that even the union of all $\mathcal{K}_J(\tilde{S})$'s can not be the entire $\mathcal{C}(\tilde{S})$, the manifold $\tilde{S}$ doesn't admit symplectic generic complex structure by uniqueness of symplectic canonical class on $S$ (see the beginning of Section \ref{section:symp}). Due to the examples given by \cite{cascinipanov} and our Theorem \ref{thm:mainthm1}, we would like to update the question here.

\begin{question}[Refinement of \cite{spaceofsymp} Question 4.6]\label{question:refine}
     Which manifolds $X$, underlying a minimal K\"{a}hler surface with $b_2^+(X)=1$, allow for demonstrating the absence of a symplectic generic complex structure on the blowup manifold $X\#n\overline{\mathbb{CP}^2}$ for some $n\geq 0$?
\end{question}

\begin{rmk}
    Question \ref{question:refine} is open for $X=\mathbb{CP}^2$, which is related to some famous conjectures on rational surfaces such as Nagata conjecture and SHGH conjecture. We refer to \cite{Biranduke} for the exploration in this direction, including the relation to symplectic packing problems.
\end{rmk}

Now we briefly explain the reason why the blowup of Enriques surface could serve as an example of non-K\"{a}hler symplectic manifold. The proof of Theorem \ref{thm:mainthm1} will rely on the abundance of elliptic fibrations on Enriques surface. Roughly speaking, this means when the manifold $S$ is equipped with any complex structure $J$, one can always think of $S$ covered by an $S^2$-family of $J$-holomorphic tori. The proper transforms of these tori in the blowup manifold $(\tilde{S},\tilde{J})$ will provide the obstructions for a cohomology class of $\tilde{S}$ being K\"{a}hler. On the other hand, when the complex structure $\tilde{J}$ is deformed into a generic almost complex structure, these proper transforms will not persist by Gromov's theory in symplectic geometry. Hence, the classes of these proper transforms don't provide obstructions for a class being symplectic, which cause the distinction between K\"{a}hler cone and symplectic cone.

To make the previous illustration more precise, we need a quantitative characterization of the non-K\"{a}hlerness which is encoded by the relations among several important invariants coming from algebraic geometry. Assume $L$ is a nef line bundle on an Enriques surface $S$. These invariants, introduced by various individuals at different times and for distinct purposes, are (see Section \ref{section:threeinvariants} for more precise definitions)
\begin{itemize}
    \item $\Phi$-invariant introduced in \cite{Cossec} $$\Phi(L):=\inf \{L\cdot C\,|\,C^2=0,C\in \text{Pic}(S)\text{ is effective}\}.$$
    \item Seshadri constant introduced in \cite{Demailly}
    \[\varepsilon_s(S,L):=\inf\{\frac{L\cdot C}{\nu (C,s)}\,|\,C \text{ is an irreducible curve passing through }s\},\]
    where $s$ is a point on $S$.
    \item  Algebraic capacities introduced in \cite{Worm1,Worm2}
    $$c_{k}^{\text{alg}}(S,L):=\inf \{L\cdot C\,|\,C^2-K_S\cdot C\geq 2k,C\in \text{Pic}(S)\text{ is nef}\},$$where $k$ is an integer and $K_S$ denotes the canonical class.
\end{itemize}
Although they are defined for line bundles, in Section \ref{section:smooth} we will have a polytope $\Delta$ (polyhedra cone $\tilde{\Delta}$) of dimension nine (ten) which is a fundamental domain under the diffeomorphism group and $\mathbb{R}^+$-rescaling (without $\mathbb{R}^+$-rescaling) actions on $H^2(S;\mathbb{R})$, where all these invariants can be viewed as functions and compared. This will be convenient for the computational purposes. Our Theorem \ref{thm:mainthm1} actually follows from the following comparison results.
\begin{theorem}\label{thm:mainthm2}
  Suppose the ample cone of the Enriques surface contains the interior of the $\tilde{\Delta}$, then the $\Phi$-invariant, Seshadri constant and algebraic capacities can be naturally viewed as functions over $\Delta$ (or $\tilde{\Delta}$) such that $$\sup _{s\in S}\varepsilon_s\leq 2c_0^{\text{alg}}=2\Phi.$$
\end{theorem}

We end this introduction section by a digression on the $b_2^+>1$ case.
\begin{rmk}
    When $b^+_2(X)>1$ each $\mathcal{K}_J(X)$ is contained in a proper linear subspace $H_J^{1,1}(X;\mathbb{R})\subseteq H^2(X;\mathbb{R})$. Although it doesn't make sense to talk about symplectic generic complex structure, one can still ask whether $\mathcal{C}(X)=\mathcal{K}(X)$ or not. Note that by \cite{Draghici} and the fact that Kodaira dimension $-\infty$ manifolds have $b_2^+=1$, we only need to consider Kodaira dimension $0$ or $1$. For torus bundle over torus, \cite{GeigesT2bundle} fully determined the symplectic cone. \cite{T4} studied the unobstructedness of ball packing for $T^4$ and obtained the information about $\mathcal{C}(X)$ for one-point blowup of $T^4$. This was further generalized by \cite{EV1} to manifolds admitting IHS hyperk\"{a}hler structures. For elliptic surfaces $E(m)$ with $m>2$, \cite{Hamiltonconeofellipticsurface} obtained some partial information about $\mathcal{C}(X)$. See also \cite{LiUsher} for their effort in understanding $\mathcal{C}(X)$ by doing negative inflations. On the other hand, much less is known for $\mathcal{K}(X)$. \cite{T4} verified $\mathcal{K}(X)$ is strictly smaller than $\mathcal{C}(X)$ for one-point blowup of $T^4$ by utilizing results regarding the Seshadri constant on  principally polarized abelian surfaces . 
\end{rmk}

\section{Enriques surface as smooth manifold}\label{section:smooth}
\subsection{Some standard facts and notations}

Of course, we have to firstly give the most indispensable definition in this paper.
\begin{definition}\label{def:enriques}
    An Enriques surface is a compact complex surface satisfying 
    \begin{itemize}
        \item (regularity) $q:=\rm{dim}$$H^1(\mathcal{O})=0$;
        \item (non-trivial canonical bundle) $K \neq \mathcal{O}$;
        \item ($2$-torsion canonical bundle) $K^{\otimes 2}=\mathcal{O}$.
    \end{itemize}
\end{definition}
We recall some standard facts and refer to \cite{BHPV,FMbook,IntrotoEnr} for more details. Firstly, due to the fundamental result by Horikawa (\cite{Horikawa}), all Enriques surfaces are deformation equivalent as complex surfaces, hence oriented diffeomorphic. We thus denote by $S$ the underlying closed oriented smooth $4$-manifold. There are numerous methods available for constructing such a manifold. Here we only list few of them from the topological perspective.

\begin{itemize}
    \item (\cite{GSbook}) Quotient of the complete intersection $S(2,2,2)\subseteq \mathbb{CP}^5$, namely the K3 surface defined by quadratic polynomial equations $p_i(z_0,z_1,z_2)+q_i(z_3,z_4,z_5)=0$ for $i=1,2,3$, by the $\mathbb{Z}_2$-action of the free involution $[z_0:z_1:z_2:z_3:z_4:z_5]\mapsto [z_0:z_1:z_2:-z_3:-z_4:-z_5]$.
    \item (\cite{GSbook}) Logarithmic transform $E(1)_{2,2}$ of the rational elliptic surface  $E(1)=\mathbb{CP}^2\#9\overline{\mathbb{CP}^2}$.
    \item (\cite{UsherKod}) Fiber sum $E(1)\#_{T'=T} S^2\times T^2$ along the embedded tori $T'\subseteq E(1)$ and $T\subseteq S^2\times T^2$ representing the anti-canonical class.  
    \item (\cite{SeppiKodfibersum}) Rational blowdown of a $(-4)$-sphere within the homology class $6H-\sum_{i=1}^{10} E_i$ in $\mathbb{CP}^2\#10\overline{\mathbb{CP}^2}$, where $H$ is the line class and $E_i$'s are exceptional classes.
    \item (\cite{LeungSymington}) Torus fibration over the base $\mathbb{RP}^2$ with $12$ nodal singular fibers.
\end{itemize}

\begin{rmk}
    Note that \cite[Lemma 3.2]{cascinipanov} shows any complex surface {\bf homeomorphic} to $(T^2\times S^2)\#\overline{\mathbb{CP}^2}$ must be biholomorphic to the blowup of a minimal elliptic ruled surface. However, the same statement doesn't hold for Enriques surface since \cite{fakeenriques} shows that for relatively prime, odd, positive integers $(p,q)\neq (1,1)$, the logarithmic transforms $E(1)_{2p,2q}$ of the rational elliptic surface are complex surfaces homeomorphic but not diffeomorphic to $S$ (which are called the fake Enriques surfaces). 
\end{rmk}

\begin{rmk}
    If we reverse the orientation of $S$ to get $\hat{S}$, then it is neither complex nor symplectic. On the one hand, from $c_1^2(\hat{S})=3\sigma(\hat{S})+2e(\hat{S})=3\cdot 8+2\cdot 12=48>36=3c_2(\hat{S})$ one sees that $\hat{S}$ can not be complex by Bogomolov–Miyaoka–Yau inequality. On the other hand, as we will see later, $S$ contains embedded $(-2)$-spheres so that $\hat{S}$ will contain embedded $(+2)$-spheres. Then by \cite[Proposition 5.1]{FSimmersed} or \cite[Proposition 1]{Kotschickori} all Seiberg-Witten invariants of $\hat{S}$ vanish, which is impossible for a symplectic manifold by \cite{Taubesmoreconstraints}.
\end{rmk}
By the previous two remarks, we emphasize that throughout our discussions, by $S$ we mean the {\bf smooth oriented} manifold not priorly equipped with any geometric structure.
The basic topological invariants of $S$ are summarized by the following chart.
\begin{center}
\begin{tabular}{ |c|c|c|c|c|c| } 
 \hline
 $\pi_1(S)$ & $H^2(S;\mathbb{Z})$ & $Q_S$ & $b_2^+(S)$ & $e(S)$ & $\sigma(S)$ \\ 
  \hline
 $\mathbb{Z}_2$ & $\mathbb{Z}^{10}\oplus \mathbb{Z}_2$ & $-E_8\oplus U$ & $1$ & $12$ & $-8$ \\ 
  \hline

\end{tabular}
\end{center}

The rank $10$ even unimodular lattice $-E_8\oplus U$, which is the intersection form of $S$, is called the \emph{Enriques lattice}. Following the notation of \cite{Cossec}, we always choose a basis $\{r_0,r_1,\cdots,r_7,s_1,s_2\}$  for this lattice, where $r_0,r_1,\cdots,r_7$ all have square $-2$ and form a basis of $-E_8$ with Dynkin diagram

\begin{center}
    \dynkin[labels={r_1,r_0,r_2,r_3,r_4,r_5,r_6,r_7},
edge length=.75cm]E8
\end{center}
and $s_1,s_2$ are the basis for the hyperbolic plane $U$ with $s_1^2=s_2^2=0$, $s_1\cdot s_2=1$. Let $I_{1,10}$ be the unimodular lattice of rank $11$ coming from the intersection form of $\mathbb{CP}^2\#10\overline{\mathbb{CP}^2}$. Take the basis $\{l_0,l_1,\cdots,l_{10}\}\subseteq I_{1,10}$ such that $l_0^2=1,l_1^2=\cdots=l_{10}^2=-1$ and $l_i\cdot l_j=0$ for all $i\neq j$. We use $k$ to denote the element $-3l_0+l_1+\cdots+l_{10}\in I_{1,10}$. Let $-\mathds{1}_e$ be the rank $1$ lattice with $e^2=-1$. Then there is an isomorphism $$\Psi : -\mathds{1}_e\oplus -E_8\oplus U\rightarrow I_{1,10}$$ by sending
\[e\mapsto -k,\]
\[r_0\mapsto l_0-l_1-l_2-l_3,\]
\[r_i\mapsto l_i-l_{i+1},\,1\leq i\leq 7,\]
\[s_1\mapsto 3l_0-l_1-\cdots-l_8-l_9,\]
\[s_2\mapsto 3l_0-l_1-\cdots-l_8-l_{10}.\]
We also let $r_8:=l_8-l_9$ and $r_9:=l_9-l_{10}$ be two elements in the Enriques lattice under the above identification. After tensoring with $\mathbb{R}$, this isomorphism $\Psi$ can be upgraded into an isomorphism $$\Psi_{\mathbb{R}}:H^2(\tilde{S};\mathbb{R})\rightarrow I_{1,10}\otimes_{\mathbb{Z}}\mathbb{R}$$ where the subspace $H^2(S;\mathbb{R})\subseteq H^2(S;\mathbb{R})\oplus \mathbb{R}e=H^2(\tilde{S};\mathbb{R})$ is sent to the orthogonal complement of $k$. Define the \emph{positive cone} 
\[\mathcal{P}_{\tilde{S}}:=\{x^2>0\,|\,x\in H^2(\tilde{S};\mathbb{R})\}.\]
It has two connected components and one of them contains the class $l_0$ under the identification $\Psi_{\mathbb{R}}$ which we will denote by $\mathcal{P}_{\tilde{S}}'$ and call it the \emph{forward cone}. Similarly, for $H^2(S;\mathbb{R})$ we use $\mathcal{P}_S$ to denote its positive cone and $\mathcal{P}_S'$ its component that is contained in $\mathcal{P}_{\tilde{S}}'$.

The manifold $S$ has a universal covering space $\pi:T\rightarrow S$. When the manifold $T$ is equipped with a complex structure, it's called a \emph{K3 surface}. The basic topological invariants for $T$ are the following.
\begin{center}
\begin{tabular}{ |c|c|c|c|c|c| } 
 \hline
 $\pi_1(T)$ & $H^2(T;\mathbb{Z})$ & $Q_T$ & $b_2^+(T)$ & $e(T)$ & $\sigma(T)$ \\ 
  \hline
 $0$ & $\mathbb{Z}^{22}$ & $2(-E_8)\oplus 3U$ & $3$ & $24$ & $-16$ \\ 
  \hline
\end{tabular}
\end{center}

We denote by $\iota$ the covering involution on $T$ so that $S=T/\iota$. Then the action $\iota^*$ on $H^2(T;\mathbb{Z})$ is given by
\[\iota^*(x\oplus y\oplus z_1\oplus z_2\oplus z_3)=y\oplus x\oplus -z_1\oplus z_3\oplus z_2,\]
where $x,y$ belong to the two $-E_8$ components and $z_1,z_2,z_3$ belong to the three $U$ components. Let $Q_T^+$ ($Q_T^-$) be the $\iota^*$-invariant (anti-invariant) sublattice of $Q_T$. Then we see that 
\[Q_T^+=\{x\oplus x\oplus 0\oplus z\oplus z\,|\,x\in -E_8,z\in U\},\]
\[Q_T^-=\{x\oplus -x\oplus z_1\oplus z_2\oplus -z_2\,|\,x\in -E_8,z_1,z_2\in U\},\]
which are of rank $10$ and $12$ respectively. The pullback map $\pi^*:Q_S\rightarrow Q_T$ will send $x\oplus z\in -E_8\oplus U$ to $x\oplus x\oplus 0\oplus z\oplus z\in Q_T^+\subseteq Q_T$, which defines a lattice isomorphism between $Q_S$ and $Q_T^+$ but does not preserve the bilinear form. Indeed, $(\pi^*(\alpha),\pi^*(\beta))=2(\alpha,\beta)$.

\subsection{Diffeomorphism group action and normalized fundamental chamber}
By diffeomorphisms we always mean the orientation-preserving ones. Since we are interested in the K\"{a}hler cone and symplectic cone which are the domains invariant under the action by diffeomorphisms, it's important to figure out how the diffeomorphism group $\text{Diff}^+(S)$ or $\text{Diff}^+(\tilde{S})$ acts on $H^2(S;\mathbb{R})$ or $H^2(\tilde{S};\mathbb{R})$. \cite{Diffgroupofellipticsurface} investigates the image of the natural map
\[\psi_X:\text{Diff}^+(X)\rightarrow O(Q_X)\]
from the diffeomorphism group to the isometry group of the intersection form for all relatively minimal elliptic surfaces $X$ with positive Euler number. We will need the special case for Enriques surface summarized by the following theorem. Recall that (for example, see \cite[Proposition 2.4]{FM88}) an embedded sphere representing class $\alpha$ with self-intersection $(-2)$ or $(-1)$ will give a diffeomorphism whose cohomological action is the reflection $R_{\alpha}$ which maps any class $\beta$ to $\beta-2\frac{(\alpha,\beta)}{(\alpha,\alpha)}\alpha$. An isometry in $O(Q_S)$ is said to have real spinor norm one if it preserves the forward cone $\mathcal{P}'_S$. 

\begin{theorem}[\cite{Diffgroupofellipticsurface} Theorem 7, Lemma 7]\label{thm:diffaction}
    For the manifold $S$, we have the following:
    \begin{itemize}
        \item Every class in $Q_S$ of square $(-2)$ can be represented by a smoothly embedded $2$-sphere. $\text{Diff}^+(S)$ acts transitively on all $(-2)$-classes.
        \item The reflections along all $(-2)$-classes generate an index $2$ subgroup $O'(Q_S)$ of $O(Q_S)$ consisting of elements with real spinor norm one. Moreover, this subgroup can be generated by finitely many reflections along $r_0,r_1,\cdots,r_8,r_9$.
        \item There exists a diffeomorphism on $S$ inducing $-\text{id}$ on $Q_S$ and therefore, $\psi_S$ is surjective.
    \end{itemize}
\end{theorem}

It's known that for K\"{a}hler surfaces (\cite{Brussee},\cite{AlgSurandSW}) or more generally symplectic $4$-manifolds (\cite{Lismooth}) of non-negative Kodaira dimension, up to sign the classes represented by smoothly embedded $(-1)$-spheres are exactly the same as the ones represented by holomorphic or symplectic embedded $(-1)$-spheres. Consequently $\pm e$ are the only classes that can be represented smoothly embedded $(-1)$-sphere in $\tilde{S}$. As a result we have the immediate corollary:
\begin{corollary}\label{cor:diffonblowup}
    The image of $\psi_{\tilde{S}}$ splits as $O(Q_S)\times \mathbb{Z}_2$, where $\mathbb{Z}_2$ denotes the reflection along the $(-1)$-class $e$.
\end{corollary}

Observe that the positive cone $\mathcal{P}_S$ is invariant under the entire diffeomorphism group $\text{Diff}^+(S)$ and the component $\mathcal{P}'_S$ is invariant under the ones inducing isometries of $Q_S$ with real spinor norm one. Also there is the obvious $\mathbb{R}^+$-action on $\mathcal{P}'_S$ by rescaling. There is an explicit domain $\Delta$, which we call it the \emph{normalized fundamental chamber}, parametrizing the quotient of these actions. Let 
 \[\Delta:=\{(b_1,b_2,\cdots,b_{10})\in\mathbb{R}^{10}\,|\,\sum_{i=1}^{10}b_i=3,b_1+b_2+b_3\leq 1,b_1\geq b_2\geq\cdots\geq b_{10}> 0\}\]
 be the region in $\mathbb{R}^{10}$ and $\overline{\Delta}$ be its closure. Note that $\overline{\Delta}$ is a convex polytope of dimension $9$ sitting inside $\mathbb{R}^{10}$. It has $9$ vertices
 \[V_1=\frac{1}{7}(3,2,2,2,2,2,2,2,2,2)\]
 \[V_2=\frac{1}{14}(5,5,4,4,4,4,4,4,4,4)\]
 \[V_3=\frac{1}{21}(7,7,7,6,6,6,6,6,6,6)\]
 \[V_4=\frac{1}{18}(6,6,6,6,5,5,5,5,5,5)\]
 \[V_5=\frac{1}{15}(5,5,5,5,5,4,4,4,4,4)\]
 \[V_6=\frac{1}{12}(4,4,4,4,4,4,3,3,3,3)\]
 \[V_7=\frac{1}{9}(3,3,3,3,3,3,3,2,2,2)\]
 \[V_8=\frac{1}{6}(2,2,2,2,2,2,2,2,1,1)\]
 \[V_9=\frac{1}{3}(1,1,1,1,1,1,1,1,1,0)\]
and $\overline{\Delta}\setminus \Delta$ is exactly the single point $V_9$. Any point $(b_1,\cdots,b_{10})\in \overline{\Delta}$ corresponds to the class $\mathbf{b}:=l_0-\sum_{i=1}^{10}b_il_i\in H^2(S;\mathbb{R})$. Since $\overline{\Delta}$ is the convex hull of its vertices, by checking the values at all vertices one can see that $\mathbf{b}\in\mathcal{P}'_S$ for all points in $\Delta$. The next lemma follows from some general results in Coxeter theory (\cite{Coxeter}). We include a proof in our specific case for completeness, which is a minor modification of the MathSciNet reivew of \cite{KK17} by Weiyi Zhang. 
\begin{lemma}\label{lem:fundamentaldomain}
    The action of $\rm{Diff}$$^+(S)\times \mathbb{R}^+$ on $\mathcal{P}_S$ has a fundamental domain naturally parametrized by $\Delta$.
\end{lemma}

\begin{proof}
    By Theorem \ref{thm:diffaction}, the quotient of $\mathcal{P}_S$ by the entire diffeomorphism group is the same as the quotient of $\mathcal{P}'_S$ by the reflections along $(-2)$-classes $r_0,r_1,\cdots,r_8,r_9$. Note that the group $O'(Q_S)$ with the generating set $\{r_0,r_1,\cdots,r_8,r_9\}$ forms a Coxeter system. By \cite[Section 5.13]{Coxeter}, the following domain 
     \[\tilde{\Delta}:=\{(a,b_1,b_2,\cdots,b_{10})\in\mathbb{R}^{11}\,|\,\sum_{i=1}^{10}b_i=3a,b_1+b_2+b_3\leq a,b_1\geq b_2\geq\cdots\geq b_{10}\}\]
     is the fundamental domain of the Tits cone $U:=\bigcup_{g\in O'(Q_S)}g\tilde{\Delta}$ acted by $O'(Q_S)$.
     To prove our statement, it suffices to show every element in $\mathcal{P}'_S$ can be transformed into $\tilde{\Delta}$ by reflections. Now a key input from \cite[Section 5.13]{Coxeter} is that the Tits cone is convex, which allows us to only consider integral classes by rescaling and convex combination. The elements $al_0-\sum_{i=1}^{10}b_il_i\in\mathcal{P}'_S$ are characterized by 
    \[a>0,\,3a=\sum_{i=1}^{10}b_i,\,a^2>\sum_{i=1}^{10}b_{i}^2.\]
    Note that the reflections along $l_i-l_j$ or $l_0-l_i-l_j-l_k$ have the effect switching the coefficients $$(b_i,b_j)\mapsto(b_j,b_i)$$ or $$(a,b_i,b_j,b_k)\mapsto(2a-b_i-b_j-b_k,a-b_j-b_k,a-b_i-b_k,a-b_i-b_j).$$
    For elements in $\mathcal{P}_S'$ with integer coefficients, one can repeat the procedure of choosing three biggest $b_i,b_j,b_k$ and whenever $a>b_i+b_j+b_k$ applying reflection along $l_0-l_i-l_j-l_k$. Since each reflection decreases $a$ by at least $1$ and the reflection preserves the connected component $\mathcal{P}_S'$, after finitely many steps and applying reflections along $l_i-l_j$, we will get a non-increasing sequence $(b_1,\cdots,b_{10})$ with $3a=\sum_{i=1}^{10}b_i$ and $a\geq b_1+b_2+b_3$. It's easy to see from Cauchy–Schwarz inequality that these conditions force $b_{10}> 0$ since $a^2>\sum_{i=1}^{10}b_i^2$. Therefore after normalizing the coefficient $a$ to $1$ by the $\mathbb{R}^+$-action, we get an element in $\Delta$.
\end{proof}

\begin{rmk}
    The polytope $\Delta$ actually coincides with the `$\cdot c_1=0$'-part of the normalized symplectic reduced cone of $\mathbb{CP}^2\#10\overline{\mathbb{CP}^2}$ (for example, see \cite{llwtorelli}). Since cohomologous symplectic forms on Kodaira dimension $-\infty$ manifolds are diffeomorphic (\cite{spaceofsymp}), this implies there is a bijection between $\Delta$ and the space of symplectic forms $\omega$ on $\mathbb{CP}^2\#10\overline{\mathbb{CP}^2}$ satisfying $\omega\cdot c_1(\omega)=0$  modulo the action by $\text{Diff}^+(\mathbb{CP}^2\#10\overline{\mathbb{CP}^2})\times \mathbb{R}^+$. However, for Enriques surface, $\Delta$ only parametrizes the cohomology classes of symplectic forms modulo the action since it's not known whether two cohomologous symplectic forms on Enriques surface are diffeomorphic or not.
\end{rmk}

\section{Enriques surface as complex manifold}

In this section we will treat $S$ as a complex surface by equipping it with some complex structure $J$. A basic observation is that by the smooth invariance of Kodaira dimension (\cite{FriedmanQin}) we see that $S$ becomes a complex surface of Kodaira dimension $0$. Since it's also easy to see $S$ is not diffeomorphic to any other complex surface with Kodaira dimension $0$ in the list of Enriques-Kodaira classification, any complex structure will satisfy the conditions in Definition \ref{def:enriques}. Similar argument shows that any complex structure on the smooth manifold $\Tilde{S}:=S\#\overline{\mathbb{CP}^2}$ comes from the complex blowup of some Enriques surface $(S,J)$. Since $p_g(S)=0$ implies the complex surface must be algebraic and $q(S)=0$ implies the Picard variety $\text{Pic}_0(S)$ is trivial, we can naturally identify the Picard group $\text{Pic}(S)$ with $H^2(S;\mathbb{Z})$ by taking the first Chern class. The group of numerically equivalent divisor classes $\text{Num}(S)$ is thus naturally identified with $H^2(S;\mathbb{Z})/\text{Tor}=Q_S$.
\subsection{Some standard facts and notations}
We briefly review some standard facts about Enriques surface from \cite{BHPV,enriquesbook} and fix some notations.

Let $K$ be the canonical class which is the torsion element in $\text{Pic}(S)=H^2(S;\mathbb{Z})$. For any class $D\in H^2(S;\mathbb{Z})$ we denote by $|D|$ its linear system when viewing $D$ as a divisor and by $h^i(D)$ the dimension of the cohomology groups $H^i(S,\mathcal{O}_S(D))$ where $\mathcal{O}_S(D)$ is the associated line bundle on $S$. $D$ is called \emph{effective} (or sometimes $J$-effective if we want to address the complex structure $J$) if $h^0(D)=\text{dim}|D|+1\geq 1$ and \emph{nef} (or $J$-nef) if it pairs non-negatively with any effective class. We say $D$ gives an \emph{elliptic pencil} if $\text{dim}|D|=1$ with no fixed component. This implies the existence of an $S^2$-family of genus zero $J$-holomorphic curves in class $D$ which covers $S$ and all but finitely many of them are smooth and connected tori. The basic tools from complex geometry are the Riemman-Roch formula and Serre duality
\[h^0(D)-h^1(D)+h^2(D)=\frac{D(D-K)}{2}+\chi_h(\mathcal{O}_S)=\frac{D^2}{2}+1,\]
\[h^i(D)=h^{2-i}(K-D)\text{ for }i=0,1,2,\]
from which we obtain the inequality
    \begin{equation}\label{equ:RR}
        h^0(D)+h^0(K-D)\geq \frac{D^2}{2}+1.
    \end{equation}
    The following lemma can be found in \cite[Section \Romannum{8}.17]{BHPV} or \cite[Section 2.2]{enriquesbook}. We give a sketch of the proof.
    \begin{lemma}\label{lem:2D}
        Let $D$ be a primitive class in $H^2(S;\mathbb{Z})$ with $D^2=0$ and $J$ is an arbitrary complex structure on $S$. Then either $D$ or $-D$ is $J$-effective. If $D$ is both $J$-effective and $J$-nef, then either $D$ or $2D$ gives an elliptic pencil.\footnote{Actually \cite[Proposition 2.2.8]{enriquesbook} shows that the primitive assumption can guarantee $|D|$ has dimension $0$ and $|2D|$ has dimension $1$. But we don't need this result.}
    \end{lemma}
    \begin{proof}
        By (\ref{equ:RR}), if $D$ is not effective then $K-D$ must be effective. Since $(S,J)$ is always algebraic, we have $-(K-D)=K+D$ can not be effective. Now apply (\ref{equ:RR}) again to $-D$, we obtain that $-D$ is effective. Now suppose $D$ is both effective and nef. Since we also imposed the condition that $D$ is a primitive class, $D$ will be represented by a reduced by possibly reducible $J$-holomorphic curve which is an \emph{elliptic configuration} (in the language of \cite{BHPV}) or \emph{indecomposable of canonical type} (in the language of \cite{enriquesbook}). This means the configuration is from the list of Kodaira's table in his classification of singular elliptic fibres. If $D$ doesn't give an elliptic pencil, then $h^0(D)=1$ and by Riemann-Roch formula, we have $h^1(D)=0$. The short exact sequence of sheaves
        \[0\rightarrow \mathcal{O}_S(D)\rightarrow \mathcal{O}_S(2D)\rightarrow \mathcal{O}_D(2D)\rightarrow 0\] and its induced long exact sequence of cohomology shows that \[h^0(2D)=1+\text{dim} H^0(D,\mathcal{O}_D(2D)).\] Note that when $D$ is represented by a smooth elliptic curve, it's easy to see from adjunction formula that 
        \[\mathcal{O}_D(2D)\cong \mathcal{O}_D(2K+2D)\cong \mathcal{O}_D(D+K)^{\otimes 2}\cong \omega_D^{\otimes 2}\cong \mathcal{O}_D^{\otimes 2}\cong \mathcal{O}_D.\]
        Here $\omega_D$ denotes the dualizing sheaf on the curve $D$ which is trivial for a smooth elliptic curve. \cite{BHPV} or \cite{enriquesbook} also shows that this is still the case when $D$ is only assumed to be an elliptic configuration of indecomposable of canonical type. And as a result, $\text{dim} H^0(D,\mathcal{O}_D(2D))=1$ always holds true under our assumptions. So the linear system $|2D|$ has dimension $1$ and it doesn't contain fixed component since all the components must have self-intersection $(-2)$. 
    \end{proof}

\subsection{Period domain and Torelli theorem}
Now we review some knowledge about period domain of Enriques surface from \cite{Horikawa} and \cite{Namikawa}. Firstly note that when $S$ is given a complex structure $J$, the universal covering $\pi:T\rightarrow S$ and the involution $\iota:T\rightarrow T$ will become holomorphic once we equip $T$ with the complex structure $\pi^*J$. The complex surface $(T,\pi^*J)$ is then a K3 surface in the sense that it has trivial canonical line bundle and the irregularity $q=0$. The trivialization of the canonical line bundle offers a holomorphic $2$-form $\Omega$ on $T$ which is unique up to a constant rescaling. Note that $\iota^*\Omega$ is still a holomorphic $2$-form which must be either $\Omega$ or $-\Omega$ since $\iota^*$ has only $\pm 1$ as its eigenvalues. But the former case can never happen since that will imply the existence of a non-trivial holomorphic $2$-form on the quotient surface $(S,J)$ contradicting the condition $p_g(S)=0$. As a result, if we view $Q^-_{T,\mathbb{C}}:=Q_T^-\otimes_{\mathbb{Z}}\mathbb{C}$ as a complex linear subspace of $H^2(T;\mathbb{C})$, then the class $[\Omega]$ must belong to $Q^-_{T,\mathbb{C}}$. Moreover, since $\Omega$ is a $(2,0)$-form, it would satisfy $[\Omega]\cdot[\Omega]=0$ and $[\Omega]\cdot[\bar{\Omega}]>0$. Therefore, any complex structure $J$ on $S$ defines a point in the \emph{period domain}
\[\mathcal{D}:=\{(v)\in \mathbb{P}(Q^-_{T,\mathbb{C}})\,|\,(v,v)=0,(v,\bar{v})>0\}.\]
$\mathcal{D}$ is an open subset of a smooth complex projective variety of dimension $10$, on which the isometry group $\Gamma:=O(Q_T^-)$ naturally acts. However, not all the points in $\mathcal{D}$ can be realized as the period. Applying Riemman-Roch to the K3 surface, if $l\in H^{1,1}(T;\mathbb{Z})$ is a $(-2)$-class, then either $l$ or $-l$ is effective. The condition $l\in H^{1,1}(T;\mathbb{Z})$ is equivalent to $l\cdot[\Omega]=l\cdot[\bar{\Omega}]=0$. Also note that an effective class can never be $\iota$-anti-invariant. We then see that all the periods which can be realized must be contained in the following smaller subset of $\mathcal{D}$ removing countably many hyperplanes:
\[\mathcal{D}_0:=\{(v)\in \mathcal{D}\,|\,(l,v)\neq 0\text{ for all } l\in Q_{T}^-\text{ with }l^2=-2\}.\]
The isometry group $\Gamma$ also acts this smaller subset. The significant achievement in algebraic geometry is the following Torelli theorem for Enriques surface, which is proved using Torelli theorem for K3 surface.
\begin{theorem}[\cite{Horikawa,Namikawa}]\label{thm:torelli}
    There is a bijection between the space of all complex structures on $S$ modulo the action by $\text{Diff}^+(S)$ with $\mathcal{D}_0/\Gamma$.
\end{theorem}

\begin{rmk}
    We just rephrase the statement in most literature of algebraic geometry using the terminology `marked Enriques surface' by fixing the manifold $S$ and modulo the diffeomorphism group action. The moduli space $\mathcal{D}_0/\Gamma$ can be given the structure of a connected quasiprojective variety of dimension 10.
\end{rmk}

By adjunction formula, an irreducible curve $D$ on Enriques surface must be a smooth rational $(-2)$-curve. If such a curve appears, then $p^{-1}(D)$ must be a disjoint union of two smooth rational $(-2)$-curves in $T$ since $\pi$ is unramified. Denote by $\alpha,\beta$ the classes of these two curves. They will satisfy $\iota^*(\alpha)=\beta$, $(\alpha,\beta)=0$. So $l:=\alpha-\beta$ will be a $(-4)$-class in $Q_T^-$ and we call such a $(-4)$-class \emph{of even type} (following \cite{Namikawa}). Note that for effective classes $\alpha$ and $\beta$, $[\Omega]\cdot \alpha=[\Omega]\cdot \beta=0$ and thus $[\Omega]\cdot l=0$. As a consequence, if we further zoom into a smaller subset of $\mathcal{D}_0$ removing countably hyperplanes determined by $(-4)$-classes:
\[\mathcal{D}_{gen}:=\{(v)\in\mathcal{D}_0\,|\,(v,l)\neq 0, l\in Q_T^-,l^2=-4, \text{ is of even type}\}.\]
We say a complex structure $J$ on $S$ is \emph{general} if its period point belong to this $\mathcal{D}_{gen}$. So a general Enriques surface doesn't contain any smooth $(-2)$-curve.

\subsection{$\Phi$-invariant, Seshadri constant and algebraic capacity}\label{section:threeinvariants}

The goal of this section is to show the relations among $\Phi$-invariant introduced by Cossec (\cite{Cossec}), Seshadri constant introduced by Demailly (\cite{Demailly}) and algebraic capacities introduced by Wormleighton (\cite{Worm1,Worm2}), all of which can be naturally viewed as functions over the normalized fundamental chamber $\Delta$ introduced in Section \ref{section:smooth}. From these relations we will obtain the information about the K\"{a}hler cone of $\tilde{S}$.

Recall that any $(b_1,\cdots,b_{10})\in \Delta$ satisfies \[3=b_1+b_2+\cdots +b_{10},1\geq b_1+b_2+b_3, b_1\geq b_2\geq \cdots \geq b_{10}>0\]
which gives rise to the class $\mathbf{b}:=l_0-\sum_{i=1}^{10}b_il_i\in H^2(S;\mathbb{R})$. The $\Phi$-invariant is defined for an Enriques surface $(S,J)$ and a big line bundle $L\in \text{Pic}(S)$:
\[\Phi_J(L):=\inf\{L\cdot F\,|\,F^2=0,F\text{ is } J\text{-effective},F\in\text{Pic}(S)\}.\]
Since either $F$ or $-F$ is effective when $F^2=0$ (Lemma \ref{lem:2D}), one can immediately see the independence of the complex structure $J$. It can be defined over $\Delta$ by
\[\Phi(b_1,\cdots,b_{10}):=\inf \{|\mathbf{b}\cdot F|\,|\,F^2=0,F\in\text{Num}(S)\}.\]

\begin{proposition}\label{prop:phi=b10}
    $\Phi(b_1,\cdots,b_{10})=b_{10}$.
\end{proposition}
\begin{proof}
    We only need to show the class $s_1=3l_0-l_1-\cdots-l_9$ of square $0$ is the best minimizer for $\Phi$ over the polytope $\overline{\Delta}$. If $F=cl_0-\sum_{i=1}^{10}d_il_i\in Q_S$ of square $0$ is another class such that $|\mathbf{b}\cdot F|<|\mathbf{b}\cdot s_1|$, then since $\overline{\Delta}$ is the convex hull of its vertices $V_1,\cdots,V_9$, there must be some $V_i$ where the pairing with $F$ is smaller than $s_1$. We will show this can not happen case by case. Let $\mathbf{v}_i$ be the class in $H^2(S;\mathbb{R})$ corresponding to $V_i$. Note that we have the conditions $c^2=\sum_{i=1}^{10}d_i^2$ and $3c=\sum_{i=1}^{10}d_i$ and may further assume $c>0$ without loss of generality.
    \begin{itemize}
        \item At $V_1$: if $\mathbf{v_1}\cdot F<\mathbf{v}_1\cdot s_1$, then
        \[\frac{1}{7}(c-d_1)=c-\frac{1}{7}(3d_1-2\sum_{i=2}^{10}d_i)<\frac{2}{7}\] and we must have $d_1=c-1$. By Cauchy-Schwarz inequality, 
        \[c^2-(c-1)^2=\sum_{i=2}^{10}d_i^2\geq \frac{1}{9}(\sum_{i=2}^{10}d_i)^2=\frac{1}{9}(3c-(c-1))^2.\]
        This will imply $c=1,2$ which are impossible by simple observation.
        \item At $V_2$: if $\mathbf{v_2}\cdot F<\mathbf{v}_2\cdot s_1$, then
       \[\frac{1}{14}(2c-d_1-d_2)=c-\frac{1}{14}(5d_1-5d_2-4\sum_{i=3}^{10}d_i)<\frac{4}{14}\] and we must have $d_1+d_2\geq 2c-3$ so that $d_1$ or $d_2$ is $\geq c-1$. Now by the previous case we know this can not happen.
        \item At $V_k$ for $3\leq k\leq 9$: from the pattern of $V_k$ we see that if $\mathbf{v_k}\cdot F<\mathbf{v}_k\cdot s_1$, then 
        \[\frac{\sum_{i=k+1}^{10}d_i}{3(10-k)}=c-\frac{(10-k)\sum_{i=1}^kd_i+(9-k)\sum_{i=k+1}^{10}d_i}{3(10-k)}<\frac{9-k}{3(10-k)},\]
        which would imply at least one of $d_i$ is negative. However, by Cauchy-Schwarz inequality it's impossible to have nine numbers whose sum is $\geq 3c+1$ but the sum of square is $\leq c^2-1$. So these cases would not happen.
    \end{itemize}
\end{proof}

Now we introduce the algebraic capacities. The $k$-th algebraic capacity is defined for $(S,J)$ and a nef line bundle $L$
\[c_{k,J}^{\text{alg}}(L):=\inf \{L\cdot F\,|\,F^2\geq 2k, F\text{ is }J\text{-nef}, F\in \text{Pic}(S)\}.\]
The actual definition for algebraic capacities has to replace $F^2$ term by $F^2-K\cdot F$, which is the (real) Gromov-Taubes dimension introduced in Section \ref{section:symp}. Since we are dealing with Enriques surface, they are certainly equal. This notion has been generalized in \cite{tamecapacity} to the triple $(X,\omega,J)$ of a symplectic $4$-manifold with a tamed almost complex structure. It's shown in \cite{tamecapacity} that when $X$ is a rational manifold, the capacities \emph{only} depend on the cohomology class $[\omega]$ but not on the tamed almost complex structures. Now we proceed to prove a similar result for Enriques surface, but only for $k=0$ and complex structures. 

\begin{lemma}\label{lem:nefcone}
    For any complex structure $J$ on $S$, there exists a self-diffeomorphism $f$ on $S$ such that the ample cone of $(S,f^*J)$ contains the interior of $\tilde{\Delta}$ (defined in the proof of Lemma \ref{lem:fundamentaldomain}).
\end{lemma}

\begin{proof}
    By possibly composing the diffeomorphism inducing $-\text{Id}$ on cohomology, we may assume the ample cone is contained in the forward cone $\mathcal{P}_S'$. By \cite[Section 5.13]{Coxeter}, an interior point of $\tilde{\Delta}$ has trivial stabilizer group under the action of $O'(Q_S)$. So its pairing with any $(-2)$-class is non-zero otherwise the stabilizer group will contain the reflection along that $(-2)$-class. Since the ample cone is open and $\tilde{\Delta}$ is the fundamental domain, we could always compose some diffeomorphism such that the ample cone contains some point in the interior of $\tilde{\Delta}$, which pairs positively with all effective $(-2)$-classes. By connectedness of $\tilde{\Delta}$ we see that the entire interior of $\tilde{\Delta}$ must pair positively with all effective $(-2)$-classes. The result then follows from Nakai-Moishezon criterion.
\end{proof}

By the above lemma, for any complex structure $J$ on $S$, there is another $J'$ with the same period as $J$ such that the nef cone of $(S,J')$ contains the closure of $\tilde{\Delta}$. It thus makes sense to extend $c^{\text{alg}}_{k,J'}$ as a function over $\Delta$ by substituting the integral class $L$ in the original definition with the real class $\mathbf{b}$. 

\begin{proposition}\label{prop:phi=algcapa}
If the ample cone of $(S,J)$ contains the interior of $\tilde{\Delta}$, then  $\Phi=c_{0,J}^{\rm{alg}}$ as functions on $\Delta$.    
\end{proposition}
\begin{proof}
    We need to show the class $s_1=3l_0-l_1-\cdots-l_9$ is also the best minimizer for $c_{0,J}^{\rm{alg}}$ over the entire $\Delta$. By Proposition \ref{prop:phi=b10}, this amounts to show that $s_1$ is $J$-nef and there is no better minimizer with positive square. For the first thing, just note that the class $s_1$ is actually in the closure of $\tilde{\Delta}$ ($\frac{1}{3}s_1$ is the vertex $V_9$). Since the ample cone contains the interior of $\tilde{\Delta}$, the class $s_1$ must be $J$-nef. For the second thing, just go through the same argument as the proof of Proposition \ref{prop:phi=b10} with the only modification by replacing the equality $c^2=\sum_{i=1}^{10}d_i^2$ with the inequality $c^2>\sum_{i=1}^{10}d_i^2$.
\end{proof}

Since a general complex structure doesn't allow any $(-2)$-curves, the following corollary is immediate.
\begin{corollary}
    $\Phi=c_{0,J}^{\rm{alg}}$ as functions on $\Delta$ for all general complex structures $J$.  
\end{corollary}

Finally we introduce the Seshadri constant. For a nef line bundle $L$ over $(S,J)$ and a point $s\in S$, the Seshadri constant is defined to be
\[\varepsilon_{J,s}(L):=\inf\{\frac{L\cdot C}{\nu (C,s)}\,|\,C \text{ is an irreducible curve passing through }s\}\]
where $\nu (C,s)$ is the multiplicity of $C$ at $s$. Again, if the ample cone of $(S,J)$ contains the interior of $\tilde{\Delta}$, then $\varepsilon_{J,s}$ can be viewed as a function on $\Delta$ by pairing with any real class $\mathbf{b}$ in the definition. Observe that since $p_g(S)=0$, by Nakai-Moishezon criterion, Corollary \ref{cor:diffonblowup} and Lemma \ref{lem:nefcone}, if we define $\mathcal{K}:\Delta\rightarrow\mathbb{R}^+$ by
\[\mathcal{K}(b_1,\cdots,b_{10}):=\sup\{t\,|\,\mathbf{b}-te\text{ can be represented by a K\"{ahler} form on } \tilde{S}\},\]
then the relation between $\varepsilon_{J,s}$ and $\mathcal{K}$ is given by
\[\mathcal{K}(b_1,\cdots,b_{10})=\sup \{\varepsilon_{J,s}(\mathbf{b})\,|\,J\text{'s ample cone contains Int}(\tilde{\Delta}),s\in S\}.\]

\begin{proposition}\label{prop:2phi}
    $\mathcal{K}\leq 2 \Phi$ as functions over $\Delta$.
\end{proposition}
\begin{proof}
    Lemma \ref{lem:2D} tells us if $F^2=0$ and $F$ is both $J$-effective and $J$-nef, then there is a $J$-holomorphic curve in class $2F$ passing through any point on $S$. Consequently, from the definition we have \[\sup \{\varepsilon_{J,s}(\mathbf{b})\,|\,s\in S\}\leq 2c_{0,J}^{\text{alg}}(\mathbf{b})\]
    for all complex structures $J$ whose ample cone contains the interior of $\tilde{\Delta}$. Now applying Proposition \ref{prop:phi=algcapa} and taking the supremum over all such complex structures $J$ for the above inequality, we obtain that $\mathcal{K}\leq 2 \Phi.$
\end{proof}

\section{Enriques surface as symplectic manifold}\label{section:symp}
In this section will treat $S$ and $\tilde{S}$ as symplectic $4$-manifolds by equipping then with any symplectic form $\omega$ and $\tilde{\omega}$ compatible with the orientation. A basic fact from \cite[Theorem 1]{LiLiucone} we should keep in mind is the uniqueness of the first Chern class. That is, $c_1(S,\omega)$ must be the torsion element in $H^2(S;\mathbb{Z})$ and $c_1(\tilde{S},\tilde{\omega})$ must be $c_1(S,\omega)\pm e$ where the sign in front of $e$ depends on the sign of $\tilde{\omega}(e)$. We always assume the cohomology class $[\omega]$ or $[\tilde{\omega}]$ lies in the connected component $\mathcal{P}_S'$ or $\mathcal{P}_{\tilde{S}}'$. Such a choice will make the invariants introduced in the sequel unambiguous. 
\subsection{Gromov-Taubes and Seiberg-Witten invariants}
For a closed symplectic $4$-manifold $(X,\omega)$ with $b_2^+=1$, we can consider its Gromov-Taubes invariant introduced in \cite{Taubescounting}. It is a function
\[Gr_{\omega}: H^2(X;\mathbb{Z})\rightarrow \mathbb{Z}\]
which is defined to be $0$ whenever $B\in H^2(X;\mathbb{Z})$ has negative complex Gromov-Taubes dimension $$d(B):=\frac{1}{2}(B^2+c_1(X,\omega)\cdot B);$$ is defined to be a sophisticated count of embedded, but may be non-connected and with multiplicities, pseudoholomorphic submanifolds passing through $d(B)$ points whose homology class is Poincar\'{e} dual to $B$ for a generic choice of $\omega$-tame almost complex structure $J$ and $d(B)$ points on $S$ when $d(B)\geq 0$ and $B\neq 0$; and is defined to be $1$ if $B=0$. Note that \cite[Theorem 1.2]{McDufflecturenotes} shows that when the almost complex structure is generic, all the components with negative self-intersection must be exceptional spheres. Therefore the non-vanishing of the invariant for a non-zero class $B$ will then imply the existence of an embedded symplectic surface representing the class $B$ with no components having self-intersection $\leq -2$. When $(X,\omega)$ is not minimal, let $\mathcal{E}_{\omega}$ be the set of classes represented by embedded symplectic exceptional spheres. To circumvent the subtlety of multiply-covered exceptional spheres, \cite{McDufflecturenotes} proposed a modified version of this invariant 
$$Gr'_{\omega}: H^2(X;\mathbb{Z})\rightarrow \mathbb{Z}$$ by assigning $Gr'_{\omega}(B)$ to $Gr_{\omega}(B')$ where 
$$B':=B+\sum_{E\in \mathcal{E}_\omega, E\cdot B<-1}(E\cdot B)E.$$

On the other hand, Seiberg-Witten theory (see \cite{Salamonbook} or \cite{Nicobook}) gives two functions
$$SW_{\omega,\pm}: H^2(X;\mathbb{Z})\rightarrow \mathbb{Z}$$ which are defined as follows. We first need a Riemannian metric $g$ and a spin$^c$ structure which gives the spinor bundles $W^{\pm}$ with $\text{det}(W^+)=\text{det}(W^-)=\mathcal{L}$. By our choice of $\omega$, there is a bijection between spin$^c$ structures and $H^2(X;\mathbb{Z})$ by associating $\mathcal{L}$ to $B:=\frac{1}{2}(c_1(\mathcal{L})-c_1(X,\omega))$. Then for a self-dual $2$-form $\eta$, the Seiberg-Witten equations are for a pair $(A,\phi)$ consisting of a connection $A$ of $\mathcal{L}$ and a section $\phi$ of $W^+$ which are given by
$$\begin{cases}
    D_A\phi=0\\
    F_A^+=iq(\phi)+i\eta,
\end{cases}$$
where $D_A$ is the Dirac operator, $F_A^+$ is the self-dual component of the curvature $F_A$ and $q:\Gamma(W^+)\rightarrow \Omega_+^2(X)$ is a canonical map. When the choice of $(g,\eta)$ is generic, the quotient of the space of solutions by $C^{\infty}(X;S^1)$ is a compact manifold $\mathcal{M}_X(\mathcal{L},g,\eta)$ of real dimension $2d(B)$. There is a principal $S^1$-bundle $\mathcal{M}^0_X(\mathcal{L},g,\eta)$ over $\mathcal{M}_X(\mathcal{L},g,\eta)$ if we take the quotient by only elements in $C^{\infty}(X;S^1)$ mapping a base point of $X$ to $1\in S^1$. Now we can get a number by pairing the maximal cup product of the fundamental class of $\mathcal{M}_X(\mathcal{L},g,\eta)$ with the Euler class of the principal $S^1$-bundle. If $b_2^+(X)=1$, the number defined above depends on the chamber where the pair $(g,\eta)$ lives. More precisely, let $\omega_g$ denote the $g$-harmonic self-dual $2$-form whose cohomology class lies in the same component as $[\omega]$ in the positive cone. Then whether this number is the value for $SW_{\omega,+}(B)$ or $SW_{\omega,-}(B)$ is determined by the sign of the discriminant 
\[\Delta_{\mathcal{L}}(g,\eta):=\int_X (2\pi c_1(\mathcal{L})+\eta)\wedge \omega_g.\]

The following results are the key properties of these invariants we will apply.
\begin{theorem}[see \cite{LiLiuWallcrossing,LiLiuIMRN}]\label{thm:liliu}
    Let $(X,\omega)$ be a closed symplectic $4$-manifold with $b_2^+(X)=1$ and $B\in H^2(X;\mathbb{Z})$, then
    \begin{itemize}
        \item (Symmetry lemma) $|SW_{\omega,+}(B)|=|SW_{\omega,-}(-c_1(X,\omega)-B)|.$
        \item (Wall crossing formula) $|SW_{\omega,+}(B)-SW_{\omega,-}(B)|=1.$
        \item (SW=Gr for $b_2^+=1$) $Gr'_\omega (B)=SW_{\omega,-}(B)$ when the set $\{E\in \mathcal{E}_{\omega}\,|\,E\cdot B<-1\}$ is finite.
        \item (Blowup formula) Let $(\tilde{X},\tilde{\omega})$ be the symplectic blowup of $(X,\omega)$ with exceptional class $e$, then $SW_{\tilde{\omega},-}(B+le)=SW_{\omega,-}(B)$ when $d(B+le)\geq0$.
    \end{itemize}
\end{theorem}

Note that for manifolds with non-negative Kodaira dimension, by \cite[Corollary 3]{Lismooth} the set $\{E\in \mathcal{E}_{\omega}\,|\,E\cdot B<-1\}$ is always finite. For rational manifolds $\mathbb{CP}^2\#n\overline{\mathbb{CP}^2}$, the subtlety around the infinity of this set was further investigated in \cite{tamecapacity}. Now we determine all the classes with non-trivial invariants for Enriques surface $S$ and its one-point blowup $\tilde{S}$.
\begin{prop}\label{prop:basicclass}
   Let $S$ and $\tilde{S}$ be equipped with the symplectic forms $\omega$ and $\tilde{\omega}$ whose cohomology classes live in $\mathcal{P}_S'$ and $\mathcal{P}_{\tilde{S}}'$ respectively. Let $B\in H^2(S;\mathbb{Z})$ and $\tilde{B}=B+le\in H^2(S;\mathbb{Z})\oplus \mathbb{Z}e= H^2(\tilde{S};\mathbb{Z})$. Denote by $B_f$ the torsion-free part of $B$. Assume further that $\tilde{\omega}\cdot e>0$. Then\begin{itemize}
       \item $Gr_{\omega}(B)\neq 0 \iff Gr'_{\omega}(B)\neq 0 \iff SW_{\omega,-}(B)\neq 0 \iff B_f\in\overline{\mathcal{P}_S'},$
       \item $Gr_{\tilde{\omega}}(\tilde{B})\neq 0 \iff B^2\geq l^2-l,l\leq 1\text{ and }B_f\in\overline{\mathcal{P}_S'},$
       \item $Gr'_{\tilde{\omega}}(\tilde{B})\neq 0 \iff SW_{\tilde{\omega},-}(\tilde{B})\neq 0 \iff B_f\in\overline{\mathcal{P}_S'} \text{ and either } B^2\geq l^2-l \text{ or } l\geq 2.$
   \end{itemize}
\end{prop}
\begin{proof}
    For the first item, the minimality of $S$ and SW=Gr in Theorem \ref{thm:liliu} give the first two equivalence. Also note that a class $B$ with non-trivial invariants must have non-negative $d(B)=\frac{1}{2}(B^2+c_1(S,\omega)\cdot B)=\frac{1}{2}B^2$ and $\omega\cdot B>0$ since $B$ can be represented by embedded symplectic surfaces, which implies $B_f\in\overline{\mathcal{P}_S'}$ by the assumption that $[\omega]\in\mathcal{P}_S'$. On the other hand, if $B_f\in\overline{\mathcal{P}_S'}$, then $\omega\cdot(-c_1(S,\omega)-B)<0$ which implies $Gr_{\omega}(-c_1(S,\omega)-B)=0$. Now by symmetry lemma and wall crossing formula in Theorem \ref{thm:liliu}, $SW_{\omega,-}(B)\neq 0$. This shows the third equivalence.\\
    For the second item, note that the condition $\tilde{\omega}\cdot e>0$ guarantees $c_1(\tilde{S},\tilde{\omega})=c_1(S,\omega)-e$. Assuming $Gr_{\tilde{\omega}}(\tilde{B})\neq 0$, then we firstly have $0\leq d(\tilde{B})=B^2-l^2+l$. If $l\geq 2$, then by positivity of intersection and $Gr_{\omega'}(e)\neq 0$ (can be seen from the blowup formula), any embedded pseudoholomorphic representative of the class $B$ must have a component of the exceptional class $e$ with multiplicity $\geq 2$, which doesn't appear for generic almost complex structures. This shows $l\leq 1$ and further implies $\{E\in \mathcal{E}_{\tilde{\omega}}\,|\,E\cdot \tilde{B}<-1\}$ is empty. So $0\neq Gr_{\tilde{\omega}}(\tilde{B})=Gr'_{\tilde{\omega}}(\tilde{B})=SW_{\tilde{\omega},-}(\tilde{B})=SW_{\omega,-}(B)$ by Theorem \ref{thm:liliu}. The first item then implies $B_f\in\overline{\mathcal{P}_S'}$. The another direction also easily follows from the similar argument.\\
    For the third item, just note that $ Gr'_{\tilde{\omega}}(\tilde{B})=Gr_{\tilde{\omega}}(\tilde{B}-le)$ when $l\geq 2$ by the definition of $Gr'$ and apply the first and the second item.
    
\end{proof}

The following corollary immediately follows from the definition of Gromov-Taubes invariant.

\begin{corollary}\label{cor:embeddedrep}
Under the same assumption of Proposition \ref{prop:basicclass}, any class $\tilde{B}=B+le\in H^2(S;\mathbb{Z})\oplus \mathbb{Z}e= H^2(\tilde{S};\mathbb{Z})$ satisfying $B^2\geq l^2-l,l\leq 0\text{ and }B_f\in\overline{\mathcal{P}_S'}$ can be represented by a connected embedded $\tilde{\omega}$-symplectic surface in $(\tilde{S},\tilde{\omega})$.
\end{corollary}
\begin{proof}
    From Proposition \ref{prop:basicclass} we know $Gr_{\tilde{\omega}}(\tilde{B})\neq 0$ so that $\tilde{B}$ can be represented by embedded symplectic surface whose components with negative self-intersection are exceptional spheres. But the condition $l\leq 0$ doesn't allow the appearance of exceptional spheres by positivity of intersection. So all the components must have non-negative self-intersection, which implies the number of components must be one since $b_2^+(\tilde{S})=1$. 
\end{proof}

\subsection{Symplectic cone of $\tilde{S}$}
In \cite{LiLiucone}, the symplectic cone has been determined for all smooth $4$-manifolds which admit symplectic structures and have $b_2^+=1$. Here we recall its proof, specializing on the manifold $\tilde{S}$. The main ingredient for constructing symplectic forms is the inflation procedure of adding suitable Thom forms near a symplectic submanifold with positive self-intersection.

\begin{lemma}[\cite{deftoiso}]\label{lem:inflation}
    Let $(X,\omega)$ be a symplectic $4$-manifold and $B\in H^2(X;\mathbb{Z})$. Assume $B^2>0$ and $B$ can be represented by a connected embedded $\omega$-symplectic surface. Denote by $B_f$ the torsion free part of $B$.  Then for all $t>0$, the cohomology class $[\omega]+t B_f$ can be realized by a symplectic form.
\end{lemma}

Combined with the calculation of the classes with non-trivial invariants from the previous section, we then have the following.
\begin{prop}
    Any cohomology class $a\in \mathcal{P}'_{\tilde{S}}$ with $a\cdot e\neq 0$ can be realized by a symplectic form on $\tilde{S}$.
\end{prop}

\begin{proof}
By rescaling and density of rational classes, without loss of generality we may assume $a$ is an integral class since the symplectic cone is open. Also, since there is the diffeomorphism coming from the reflection along the exceptional sphere in class $e$, we may further assume $a\cdot e>0$. Now pick an arbitrary integral symplectic form $\tilde{\omega}$ whose cohomology class belongs to $\mathcal{P}'_{\tilde{S}}$. Then for sufficiently large positive integer $P$, the class $Pa-[\tilde{\omega}]$ will have positive square and satisfy the conditions in Corollary \ref{cor:embeddedrep}. Thus it will be represented by a connected embedded $\tilde{\omega}$-symplectic surface and Lemma \ref{lem:inflation} can be applied to produce the symplectic form realizing the class $[\tilde{\omega}]+Pa-[\tilde{\omega}]=Pa$. By rescaling $a$ can be realized as well.
\end{proof}

For the convenience of stating our comparison theorem later, we introduce the function $\mathcal{S}:\Delta\rightarrow\mathbb{R}^+$ defined on the fundamental chamber by
\[(b_1,\cdots,b_{10})\mapsto \sup\{t\,|\,l_0-\sum_{i=1}^{10}b_il_i-te\text{ can be realized by a symplectic form}\}.\]
\begin{corollary}\label{cor:sympcone}
    $\mathcal{S}(b_1,\cdots,b_{10})=(1-\sum_{i=1}^{10}b_i^2)^{\frac{1}{2}}.$
\end{corollary}

\begin{rmk}
    Corollary \ref{cor:sympcone} has been essentially known in \cite{LiLiucone}. It's interesting to see if there is another proof without the heavy machinery of Taubes-Seiberg-Witten theory by directly showing the unobstructedness of ball packings, possibly using the techniques of almost toric fibrations of \cite{LeungSymington}.
\end{rmk}

    \section{Proof of the main theorems}
 \begin{proof}[Proof of Theorem \ref{thm:mainthm2}]
     By Proposition \ref{prop:phi=algcapa} and Proposition \ref{prop:2phi}.
 \end{proof}

\begin{proof}[Proof of Theorem \ref{thm:mainthm1}]
     By  Theorem \ref{thm:mainthm2}, Corollary \ref{cor:diffonblowup}, Proposition \ref{prop:phi=b10} and Corollary \ref{cor:sympcone}, we only need to check there exists $(b_1,\cdots,b_{10})\in \Delta$ such that $\mathcal{S}(b_1,\cdots,b_{10})>\mathcal{K}(b_1,\cdots,b_{10})$. This is equivalent to $$(1-\sum_{i=1}^{10}b_i^2)^{\frac{1}{2}}>2b_{10}.$$For example we can take 
     $$(b_1,\cdots,b_{10})=(\frac{1}{3},\frac{1}{3},\frac{1}{3},\frac{1}{3},\frac{1}{3},\frac{1}{3},\frac{1}{3},\frac{1}{3},\frac{1}{4},\frac{1}{12}).$$ See also Figure \ref{fig:cone}. 
 \end{proof}

 \begin{figure}[h]
     \centering
     \includegraphics*[width=\linewidth]{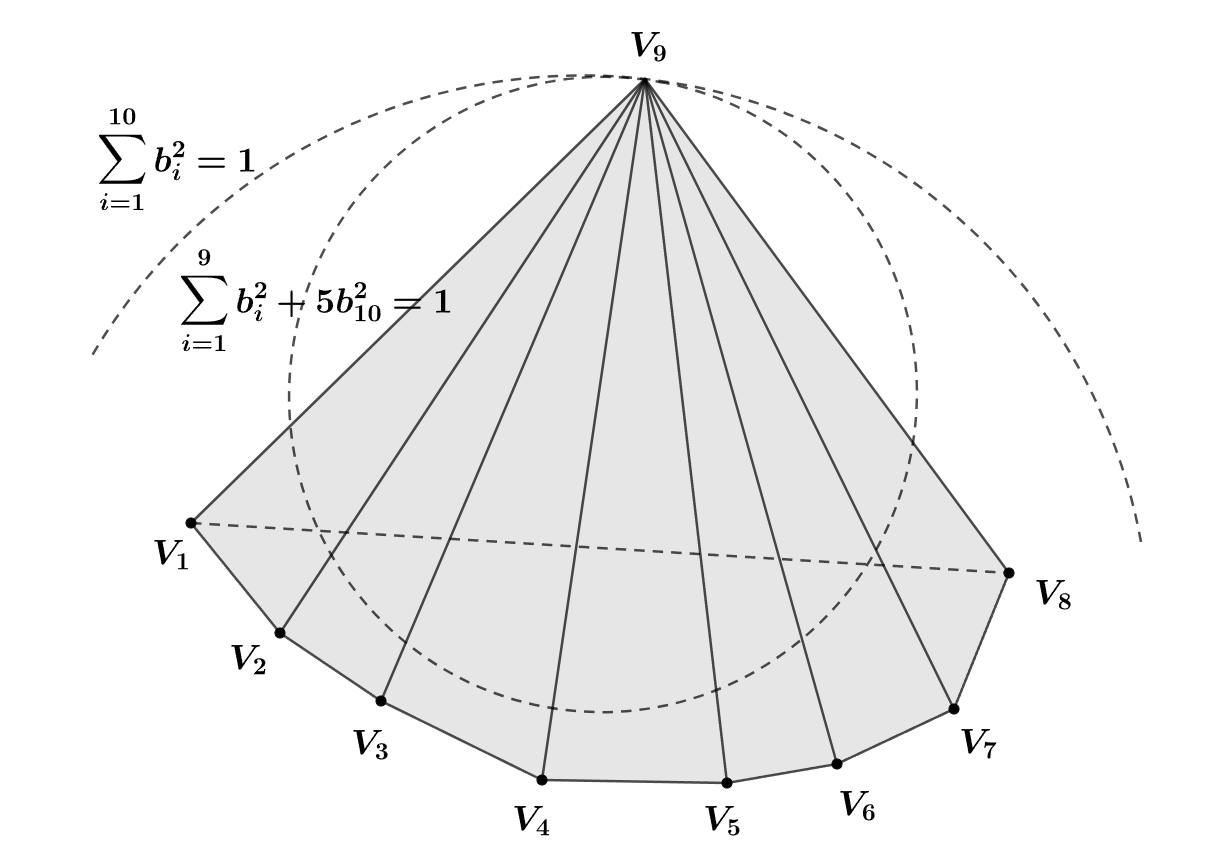}
     \caption{This is a picture for the normalized fundamental chamber $\Delta$ with $9$ vertices $V_1,\cdots.V_9$. The region enclosed by the dashed ellipse satisfies $\mathcal{K}<\mathcal{S}$, from which one can see the difference between K\"{a}hler cone and symplectic cone. }
     \label{fig:cone}
 \end{figure}


\begin{rmk}
   The recent result by \cite[Theorem 1.3]{seshadrionenriques} verifies that for a very general Enriques surface $(S,J)$, one has $\Phi=\inf _{s\in S}\varepsilon_{J,s}$. In our terminology, this means $\mathcal{K}\geq \Phi$ holds true. Therefore $\mathcal{K}$ is a function between $\Phi$ and $2\Phi$.
\end{rmk}

\bibliographystyle{amsalpha}
\bibliography{mybib}{}

\end{document}